\documentclass{amsart}

\usepackage{amsmath,amssymb}
\usepackage{graphicx}
\usepackage{xcolor}
\usepackage{hyperref}
\usepackage{bm}
\usepackage{txfonts}

\def\cY{\overline{Y}}

\def\inv{^{-1}}
\def\cO{\mathcal{O}}

\newcommand{\bQ}{\ensuremath{\mathbb{Q}}}
\newcommand{\bR}{\ensuremath{\mathbb{R}}}
\newcommand{\bZ}{\ensuremath{\mathbb{Z}}}

\newcommand{\beq}{\begin{equation}\begin{aligned}}
\newcommand{\eeq}{\end{aligned}\end{equation}}
\newcommand{\case}[2][cccccccccccccccccccccccccccccccccccccccccc]{\left\{\begin{array}{#1}#2 \\ \end{array}\right.}
\newcommand{\Eq}[1]{\begin{align}#1\end{align}}
\newcommand{\Eqn}[1]{\begin{align*}#1\end{align*}}

\newtheorem{Thm}{Theorem}[section]
\newtheorem{Lem}[Thm]{Lemma}
\newtheorem{Prop}[Thm]{Proposition}
\newtheorem{Cor}[Thm]{Corollary}
\newtheorem{Rem}[Thm]{Remark}
\newtheorem{Def}[Thm]{Definition}

\newtheorem{Ex}[Thm]{Example}

\newtheoremstyle{named}{}{}{\itshape}{}{\bfseries}{.}{.5em}{#1 #3}
\theoremstyle{named}

\def\R{\mathbb{R}}

\def\C{\mathbb{C}}

\def\cO{\mathcal{O}}

\def\a{\alpha}

\def\e{\epsilon}

\def\=>{\Longrightarrow}

\def\o+{\oplus}
\def\bo+{\bigoplus}

\def\<{\langle}
\def\>{\rangle}
\def\oo{\infty}

\def\^{\wedge}
\def\+{\dagger}

\def\inv{^{-1}}

\def\dd[#1,#2]{\frac{d#1}{d#2}}
\def\del[#1,#2]{\frac{\partial #1}{\partial #2}}
\def\over[#1]{\overline{#1}}
\def\vec[#1]{\overrightarrow{#1}}

\def\tab{\;\;\;\;\;\;}

\newcommand{\til}[1]{\widetilde{#1}}

\newcommand{\veca}[1][cccccccccccccccccccccccccccccccccccccccccc]{\begin{pmatrix}#1 \end{pmatrix}}

\newcommand{\Li}{\mathrm{Li}}
\newcommand{\newg}{\mathsf{g}}

\begin{document}

\title{Quantum Dilogarithm Identities at Root of Unity}
\author{Ivan Chi-Ho Ip and Masahito Yamazaki}
\address{Kavli IPMU (WPI), University of Tokyo, Chiba 277-8583, Japan}
\email{ivan.ip@ipmu.jp\\masahito.yamazaki@ipmu.jp}
\dedicatory{Dedicated to the memory of Kentaro Nagao}
\subjclass[2010]{13F60}

\begin{abstract}
We study the root of unity degeneration of cluster algebras and quantum dilogarithm identities.
We prove identities for the cyclic dilogarithm
associated with a mutation sequence of a quiver, and as a consequence new identities for the non-compact quantum dilogarithm at $b=1$.
\end{abstract}

\maketitle
\tableofcontents

\section{Introduction}\label{sec.intro}

Quantum dilogarithm identities \cite{ReinekePoisson,KellerOn,KashaevN} are remarkable identities satisfied by the quantum dilogarithm function. 
They are associated with a mutation sequence of a quiver, and
fit nicely with the framework of the cluster algebras 
\cite{FominZelevinsky_1,FominZelevinsky_4} and their quantization \cite{FG_Dilogarithm,FG_Ensembles}.
The simplest case of the quantum dilogarithm identity, associated with the $A_2$ quiver,
is the celebrated pentagon identity \cite{Faddeev:1993rs}. More
general quiver mutations can generate identities of increasing complexity.
These general quantum dilogarithm identities have
applications to a wide-ranging topics in mathematics and physics---the literature is too vast to be exhausted here (see e.g. references in \cite{KashaevN}), but 
to name a few includes
Donaldson-Thomas theory \cite{Kontsevich:2008fj,Nagao_dilog,KellerOn},
integrable lattice models \cite{Bazhanov:2007mh,Terashima:2012cx,Yamazaki:2012cp},
3d hyperbolic geometry \cite{Nagao:2011aa,HikamiInoue},
4d \cite{Gaiotto:2010be,Cecotti:2010fi,Xie:2012gd} and 3d \cite{Terashima:2013fg} supersymmetric gauge theories.

Quantum cluster algebra has a one deformation parameter $q$.
In this paper we analyze the case where
$q$ is a root of unity.
In this limit, the compact quantum dilogarithm function $\Psi_q(x)$ degenerates into a 
cyclic dilogarithm, and we obtain 
identities for the cyclic dilogarithm, again associated with a 
mutation sequence of a quiver.

It is known \cite{KashaevN} that identities for $\Psi_q(x)$ takes two different forms,
the tropical form (Theorem \ref{qdilog_tropical}) and the universal form (Theorem \ref{qdilog_id_universal}),
where the arguments are given respectively by the tropical quantum $y$-variables and 
the (full) quantum $y$-variables, respectively.
Our cyclic dilogarithm identities also can written in two different ways (Theorem \ref{cyclic_dilog_id} and Theorem \ref{cyclic_id_standard}),
but their arguments are given by quantum $y$-variables as well as the cyclic $y$-variables (see Definition \ref{cyclic_y_def}), where the latter can be thought of as the $N$-th root of the 
quantum $y$-variable for the modular double.

We also study the root of unity limit for the non-compact quantum dilogarithm $\Phi_b(x)$, which is discussed in a recent preprint \cite{Garoufalidis:2014ifa}. Using the results from the cyclic dilogarithm identities, in Theorem \ref{g1_identity} and Theorem \ref{Main_2} we found a whole new family of identities for the non-compact quantum dilogarithm at $b=1$, corresponding to $q=-1$, which plays a very special role in the analysis. These identities are non-trivial since we are taking integral roots of non-commuting functions. The function $\Phi_1(x)$ may also provide new classical limits of the representation theory of the quantum plane $AB=q^2 BA$ \cite{Ip1}, which has only been previously studied in the $q\to 1$ limit with the quantum dilogarithm function appearing as the matrix coefficients.

Our identities is a vast generalization of the root-of-unity degeneration of the pentagon identity \cite{BazhanovR}, which has been utilized for example in
the study of 3-manifold invariants \cite{Kashaev_6j,Garoufalidis:2014ifa}
and solvable 3-dimensional lattice models \cite{Bazhanov:1992jq}.
We would like to urge the readers of this paper to find concrete applications of 
our results in their favorite mathematics/physics problems.

This paper is organized as follows. In Section \ref{sec.qdilog} we set up our conventions and 
summarize the quantum dilogarithm identities. We then consider the root of unity limit of the 
quantum dilogarithm identities and their consequences, first in Section \ref{sec.cyclic} for the compact dilogarithm $\Psi_q(x)$
and then for the non-compact quantum dilogarithm $\Phi_b(x)$ in Section \ref{sec.cyclic2}. In the Appendix we present a new proof of the root of unity limit of the non-compact dilogarithm function $\Phi_b(x)$.

\*{\textbf{Acknowledgments.}}
This work has grown up out of a over-the-tea conversation.
We would like to thank the Kavli IPMU tea time for inspiration, 
and the WPI program (MEXT, Japan) and Kavli foundation for 
generous support.

\section{Quantum Dilogarithm Identities}\label{sec.qdilog}

In this section we review the basic of classical/quantum cluster algebras and 
the quantum dilogarithm identities. Our presentation throughout follows closely \cite{KashaevN}.
We here take $q$ to be of a generic value.

\subsection{Cluster Algebras}

\begin{Def}[Mutation and Classical $y$-variables]
Let $I$ be a finite set. 
We consider a pair (initial seed) $(B,y)$ where $B=(b_{ij})_{i,j\in I}$ is a skew symmetric integer matrix, and $y=(y_i)_{i\in I}$ is an $I$-tuple of commutative variables. Let us define the mutation of $(B,y)$ at $k$ by $\mu_k(B, y)=(B', y')$,
where
\Eq{
b'_{ij}=\case{
-b_{ij} , &\mbox{$i=k$ {\rm or} $j=k,$} \\
\begin{array}{l}
b_{ij}+[-b_{ik}]_+b_{kj}+b_{ik}[b_{kj}]_+ \\
\quad =b_{ij}+[b_{ik}]_+b_{kj}+b_{ik}[-b_{kj}]_+ , 
\end{array}
&\mbox{\rm otherwise,}}
}
\Eq{
y'_i=\case{y_k\inv , &i=k,\\
\begin{array}{l}
y_iy_k^{[b_{ki}]_+}(1+y_k)^{-b_{ki}}   \\
\quad  =y_iy_k^{[-b_{ki}]_+}(1+y_k\inv)^{-b_{ki}} , 
\end{array}
&i\neq k,}
\label{classical_y}
}
with $[x]_{+}:={\rm max}(x,0)$.
Starting from the initial seed, we can repeat the mutations
to obtain a pair $(B', y')$.
We call them seeds of $(B,y)$.
Such $y'$s take values in the semifield
\begin{align}
  \bQ_+(y):= \{\mbox{rational function of $y$ with subtraction-free expression}\} , 
\end{align}
and are called (classical) $y$-variables. 
\end{Def}

\begin{Rem}
When $B$ correspond to the incident matrix of a quiver $Q$, this mutation will be a \emph{quiver mutation} of $Q$ at the vertex $i=k$, and we have $\mu_k^2={\rm id}$. Such a $Q$ is necessarily acyclic.
\end{Rem}

\begin{Def}[Tropicalization] 

Let us define the tropical semifield by
\begin{align}
{\rm Trop}(y):=\left\{\prod_{i\in I} y_i^{a_i}\Big|a_i\in \bZ\right\}
\end{align}
with tropical sum 
\begin{align}
\prod_{i\in I} y_i^{a_i}\oplus\prod_{i\in I}y_i^{b_i} := \prod_{i\in I} y_i^{\min(a_i,b_i)}.
\end{align}
Then there exists a canonical surjective semifield homomorphism 
\begin{align}
\pi_{\rm trop}:\bQ_+(y)\to {\rm Trop}(y),
\end{align} 
defined by
\begin{align}
y_i\mapsto y_i,\;\;\;\;\;\; c\in \bQ \mapsto 1.
\end{align}
\end{Def}

\begin{Def}[$c$-vector \cite{FominZelevinsky_4}]
The $c$-vector of a cluster $y$-variable $y'\in \bQ_{+}(y)$ is the vector $c=(c_i)_{i\in I} \in \bZ^{|I|}$ such that 
\begin{align}
\pi_{\rm trop}(y_i')=\prod_{j\in I}(y_j)^{c_j}.
\end{align}
\end{Def}

\begin{Def}[Tropical Sign]  For a cluster variable $y'\in \bQ_{+}(y)$, 
we set the tropical sign $\epsilon(y')=1$ (resp. $-1$) if the components of the $c$-vector for 
$y'$ are all non-negative (resp. all non-positive).
\end{Def}

\begin{Rem}
The tropical sign is always well-defined thanks to the sign coherence proven in \cite{DWZ2,NagaoCluster,PlamondonCluster}.
\end{Rem}

\subsection{Quantum $y$-variables}

Let us next discuss the quantum version of cluster $y$-variables.

\begin{Def}[Quantum $y$-variables \cite{FG_Dilogarithm,FG_Ensembles}]\label{quantum_y_def}
Consider a pair $(B,Y)$ 
where $B=(b_{ij})_{i,j\in I}$ is a skew symmetric integer matrix, and 
$Y=(Y_i)_{i\in I}$ is an $I$-tuple of non-commutative variables such that
\begin{align}
Y_iY_j = q^{2b_{ji}}Y_jY_i.
\label{quantum_y_comm}
\end{align} 
We define the quantum mutation of $(B,Y)$ at $k\in I$ by $\mu_k(B,Y):=(B',Y')$ where $b_{ij}$ mutates the same as before, and $Y'$ is given by
\Eq{Y_i'=\case{Y_k\inv,&i=k\\
\displaystyle Y_i\prod_{m=1}^{b_{ik}}(1+q^{2m-1}Y_k),& i\neq k, b_{ik}\ge 0\\
\displaystyle Y_i\prod_{m=1}^{|b_{ik}|}(1+q^{2m-1}Y_k\inv)\inv,&  i\neq k, b_{ik}<0.
}
\label{quantum_y_mutate}
}
Starting from the initial seed $(B,Y)$, we can repeat the mutations
to obtain a pair (seed) $(B', Y')$.
The variables $Y'_i$s obtained this way are called quantum $y$-variables.
\end{Def}

\begin{Rem}
The mutation rule for the quantum $y$-variable \eqref{quantum_y_mutate} reduces to the 
mutation rule for the classical $y$-variables \eqref{classical_y} in the limit $q\to 1$,
\end{Rem}

For our purposes it is useful to consider more generally non-commutative variables 
$Y_{\alpha}$ with $\alpha\in \mathbb{Z}^{|I|}$, satisfying the relation
\begin{align}
Y_{\alpha+\beta}=q^{\langle \alpha, \beta \rangle} Y_{\alpha} Y_{\beta}, \quad
\langle \alpha, \beta \rangle=- \langle \beta, \alpha \rangle
:=\alpha^T B \beta, 
\label{quantum_torus}
\end{align}
with the identification that $Y_i^{\pm1}=Y_{\pm e_i}$ for the unit vector $e_i \in\mathbb{Z}^{|I|}$.

\subsection{Identities for $\Psi_q(x)$ and $g_b(x)$}

\begin{Def}[Period]
For any $I$-sequence $\bm{k}=(k_1,..., k_L)$, $L>1$, we consider a sequence of mutations. We fix our initial seed $(B(1), Y(1)):=(B,Y)$, and set
\begin{align}
(B(t+1),Y(t+1)):=\mu_{k_t}(B(t),Y(t)).
\end{align}
Let $\sigma$ be a permutation of $I$, we call the sequence $\bm{k}$ of mutations a $\sigma$-period if
\Eq{b_{\sigma(i)\sigma(j)}(L+1)=b_{ij}(1),\;\;\; Y_{\sigma(i)}(L+1)=Y_i(1),\;\;\;i,j\in I.
\label{eq_period}}
\end{Def}

\begin{Rem}
We can give alternative definition of the $\sigma$-period by replacing the quantum $y$-variables $Y_i(t)$
by their classical counterparts $y_i(t)$ in the definition \eqref{eq_period},
where  $(B(1), y(1)):=(B,y)$ and
\begin{align}
(B(t+1),y(t+1)):=\mu_{k_t}(B(t), y(t)).
\end{align}
The two definitions of $\sigma$-periods coincide \cite{FG_Dilogarithm}.
\end{Rem}

\begin{Def}[Compact Quantum Dilogarithm Function \cite{Faddeev:1993rs,Faddeev:1993pe}]
We define the {\it (compact) quantum dilogarithm function} $\Psi_q(x)$ by\footnote{
In some literature $q^2$ is denoted by $q$.
In our notation we always have integer powers of $q$.
}
\begin{align}
\Psi_q(x):=\frac{1}{(-q x;q^2)_{\infty}} ,
\quad
(x;q)_{\infty}: = \prod_{k=0}^{\infty} (1-q^k x).
\end{align}
\end{Def}

\begin{Thm}[Quantum Dilogarithm Identity in Tropical Form \cite{ReinekePoisson,KellerOn}]\label{qdilog_tropical}

Consider a sequence of mutation labeled by an $I$-sequence $\bm{k}$, and suppose that $\bm{k}$ is a $\sigma$-period for a permutation $\sigma$.
Let $c_t$ ($t=1\ldots, L$) be the $c$-vector of the classical $y$-variables $y_{k_t}(t)$, and let us
denote their tropical signs by $\epsilon_t$ ($t=1\ldots, L$).
We then have
\Eq{\Psi_q(Y_{\epsilon_1 c_1})^{\epsilon_1} 
\cdots 
\Psi_q(Y_{\epsilon_L c_L} )^{\epsilon_L}=1.
\label{eq.qdilog_tropical}
}

\end{Thm}

We can convert this identity to another form, with the help of the following identity \cite{KashaevN}:
\begin{align}
Y_{k_t}(t) = \mathrm{Ad}
\left[
  \Psi_q(Y_{\epsilon_1 c_1})^{\epsilon_1} \cdots
  \Psi_q(Y_{\epsilon_{t-1} c_{t-1}})^{\epsilon_{t-1} }
\right]
\left( 
  Y_{c_t}
\right) ,
\label{shuffle_id}
\end{align}
which holds for an arbitrary sign sequence $(\epsilon_1, \ldots, \epsilon_t)$.
This identity means we can shuffle the position of $\Psi_q$ as 
\begin{align}
\begin{split}
&  \Psi_q(Y_{\epsilon_1 c_1})^{\epsilon_1} \cdots
  \Psi_q(Y_{\epsilon_{L-1} c_{L-1}})^{\epsilon_{L-1} }
  \Psi_q( Y_{\epsilon_L})^{\epsilon_L} \\
&\quad  =
  \Psi_q(Y_{k_L}(L)^{\epsilon_L c_L})^{\epsilon_L}    \Psi_q(Y_{\epsilon_1 c_1})^{\epsilon_1} \cdots
  \Psi_q(Y_{\epsilon_{L-1} c_{L-1}})^{\epsilon_{L-1} },
  \label{shuffle_demonstration}
  \end{split}
\end{align}
and by repeating this we arrive at the following yet another form of the quantum dilogarithm identity:
\begin{Thm}[Quantum Dilogarithm Identity in Universal Form] \label{qdilog_id_universal}
Under the assumptions of Theorem \ref{qdilog_tropical}, we have
\Eq{\Psi_q(Y_{k_L}(L)^{\epsilon_L})^{\epsilon_L}...\Psi_q(Y_{k_1}(1)^{\epsilon_1})^{\epsilon_1}=1.
}
\end{Thm}

\begin{Ex}[$A_2$ quiver]\label{A2_example_qdilog}

Let us consider the $A_2$ quiver
\begin{align}
1\longleftarrow 2. 
\end{align}
We have $I=\{1,2\}$, $b_{ij}=\begin{pmatrix}0&-1\\1&0\end{pmatrix}$, $y=(y_1,y_2)$. 
Quantum $y$-variables $Y_1, Y_2$ satisfy $Y_1Y_2 = q^2 Y_2 Y_1$.
Let the $I$-sequence be $(1,2,1,2,1)$ with $L=5$. This is a $(12)$-period.
Indeed, we have
\begin{align}
\begin{array}{lll}Y(1) = (Y_1, Y_2) , &b(1) = \begin{pmatrix}0&-1\\1&0\end{pmatrix} ,\\
Y(2)=(Y_1\inv,Y_2(1+qY_1) ),&b(2) = \begin{pmatrix}0&1\\-1&0\end{pmatrix} , \\
Y(3)=(Y_1\inv(1+qY_2+Y_1Y_2), Y_2\inv(1+q\inv Y_1)\inv),&b(3) = \begin{pmatrix}0&-1\\1&0\end{pmatrix} , \\
Y(4)=(Y_1(1+\! qY_2+\! Y_1Y_2)\inv,q\inv Y_1\inv Y_2\inv(1+\!qY_2)),&b(4) = \begin{pmatrix}0&1\\-1&0\end{pmatrix} , \\
Y(5)=(Y_2\inv,q\inv Y_1Y_2(1+q\inv Y_2)\inv) , & b(5) = \begin{pmatrix}0&-1\\1&0\end{pmatrix} ,\\
Y(6)=(Y_2, Y_1) ,  &b(6) = \begin{pmatrix}0&1\\-1&0\end{pmatrix} ,\\
\end{array}
\end{align}
where the $c$-vectors $c_t:=c(y_{k_t}(t))$ and tropical signs $\e_t$ are computed to be
\begin{align}
\begin{array}{lll}
c_1=(1,0),&\epsilon_1=+, \\
c_2=(0,1),&\epsilon_2=+ , \\
c_3=(-1,0),&\epsilon_3=- ,\\
c_4=(-1,-1),&\epsilon_4=- ,\\
c_5=(0,-1),&\epsilon_5=- .\\
\end{array}
\end{align}
The quantum dilogarithm identity is, in the tropical form
\begin{align}
\Psi_q(Y_1)\Psi_q(Y_2)\Psi_q(Y_1)\inv \Psi_q(q\inv Y_1Y_2)\inv \Psi_q(Y_2)\inv=1,
\label{Psi_pentagon}
\end{align}
where we used $Y_{e_1+e_2} = q\inv Y_{e_1} Y_{e_2}=q\inv Y_1 Y_2$ (recall \eqref{quantum_torus}).
This is the famous {\it pentagon identity} \cite{Faddeev:1993rs} for $\Psi_q(x)$.
In the universal form, we have
\begin{align} 
\begin{split}
&\Psi_q(Y_2)\inv \Psi_q\left(q(1+qY_2)\inv Y_2Y_1\right)\inv  \\
& \quad \Psi_q\left((1+qY_2+Y_1Y_2)\inv Y_1\right)\inv \Psi_q(Y_2(1+qY_1))\Psi_q(Y_1)=1.
\end{split}
\end{align}

\end{Ex}

\begin{Ex}[$A_3$ quiver]\label{A3_example_qdilog}

Consider $A_3$ quiver, where we choose the orientation 
\begin{align}
1\longleftarrow 2 \longrightarrow 3
\end{align}
such that
\begin{equation}
  B=(b)_{ij}= \veca[0&-1&0\\1&0&1\\0&-1&0].
\end{equation}
and $y=(y_1,y_2,y_3)$. Quantum $y$-variables $Y_1, Y_2, Y_3$ satisfy
\Eq{Y_1Y_2 = q^2 Y_2 Y_1, \quad Y_3Y_2 = q^2 Y_2 Y_2, \quad Y_1 Y_3=Y_3 Y_1.}

Consider a mutation sequence $(1,3,2,1,3,2,1,3,2)$.
The quantum $y$-variables are then computed to be
\footnotesize
\begin{align}
\begin{split}
Y(1) &= (Y_1, Y_2, Y_3),\\
Y(2) &= (Y_1\inv, Y_2(1+q Y_1), Y_3),\\
Y(3)&= (Y_1\inv, Y_2 X_q, Y_3\inv),\\
Y(4)&=(Y_1\inv(1+q Y_2X_q), Y_2\inv X_{q\inv}\inv, Y_3\inv (1+q Y_2 X_q)),\\
Y(5)&=(Y_1(1+q\inv Y_2 X_q)\inv,q Y_2\inv Y_1\inv X_{q\inv}\inv (1+q\inv Y_1+q Y_2 X_q), Y_3(1+q\inv Y_2 X_q)\inv),\\
Y(6)&=(Y_1(1+q\inv Y_2 X_q)\inv,q^2 Y_2\inv Y_1\inv Y_3\inv X_{q\inv}\inv (1+q\inv Y_1+q^3 Y_2 X_q)(1+q\inv Y_3+q Y_2 X_q),\\
&\tab Y_3(1+q\inv Y_2 X_q)\inv)\\
&=(Y_1(1+q\inv Y_2 X_q)\inv, Y_1\inv Y_2\inv Y_3\inv (1+q^3 Y_2+q^2 Y_2 Y_3)(1+q Y_2+q^2 Y_2 Y_1), \\
&\tab Y_3(1+q\inv Y_2 X_q)\inv),\\
Y(7)&=(Y_1(1+q\inv Y_2 X_q)\inv(1+q Y_1\inv Y_2\inv Y_3\inv (1+q^3 Y_2+q^2 Y_2 Y_3)(1+q Y_2+q^2 Y_2 Y_1)),\\
&\tab (1+q Y_2+q^2 Y_2Y_3)\inv(1+q^3Y_2+q^2Y_2Y_1)\inv Y_3Y_2Y_1,\\
&\tab Y_3(1+q\inv Y_2 X_q)\inv(1+q Y_1\inv Y_2\inv Y_3\inv (1+q^3 Y_2+q^2 Y_2 Y_3)(1+q Y_2+q^2 Y_2 Y_1)))\\
&=(q\inv Y_3\inv Y_2\inv(1+q Y_2),  (1+q Y_2+q^2 Y_2 Y_3)\inv(1+q^3 Y_2+q^2 Y_2 Y_1)\inv Y_3 Y_2 Y_1 ,\\
&\tab q\inv Y_1\inv Y_2\inv(1+q Y_2)),\\
Y(8)&=(Y_3(1+q Y_2\inv )\inv,(1+qY_2+q^2 Y_2 Y_1)\inv Y_1,q Y_2\inv Y_1\inv(1+q Y_2)),\\
Y(9)&=(Y_3(1+q Y_2\inv )\inv, Y_2\inv, Y_1(1+q Y_2\inv )\inv),\\
Y(10)&=(Y_3, Y_2, Y_1),
\end{split}
\end{align}
\normalsize
where 
\Eq{
X_q:=(1+q Y_1)(1+q Y_3)=(1+q Y_1+q Y_3+q^2 Y_1 Y_3).
\label{Xq_def}
} The mutation sequence is therefore a $(13)$-period.
The $c$-vectors and tropical signs are given by
\begin{align}
\begin{array}{ll}
c_1 = (1,0,0),&\e_1 = +,\\
c_2 = (0,0,1),&\e_2 = +,\\
c_3= (0,1,0),&\e_3 =+ ,\\
c_4 = (-1,0,0),&\e_4 = -,\\
c_5 = (0,0,-1),&\e_5=-,\\
c_6 = (-1,-1,-1),&\e_6=-,\\
c_7= (0,-1,-1),&\e_7 = -,\\
c_8 = (-1,-1,0),&\e_8 = -,\\
c_9 = (0,-1,0),&\e_9 = -.
\end{array}
\end{align}

The quantum dilogarithm identity in the tropical form is\footnote{
This identity can also be written as
$$(Y_1;q^2)_\oo(Y_3;q^2)_\oo(Y_2;q^2)_\oo = (Y_2;q^2)_\oo (-Y_2Y_1;q^2)_\oo (-Y_2 Y_3;q^2)_\oo (Y_2 Y_1Y_3;q^2)_\oo (Y_3;q^2)_\oo.$$
}
\begin{align}
\begin{split}
&\Psi_q(Y_1)\Psi_q(Y_3)\Psi_q(Y_2)\Psi_q(Y_1)\inv\Psi_q(Y_3)\inv \\
&\qquad\Psi_q(Y_1Y_2Y_3)\inv \Psi_q(qY_2Y_3)\inv \Psi_q(qY_2Y_1)\inv \Psi_q(Y_2)\inv = 1,
\end{split}
\end{align}
This can be easily seen as a consequence of the pentagon \eqref{Psi_pentagon}.

The universal form for the quantum dilogarithm identity is given by
\begin{align}
\begin{split}
&\Psi_q(Y_2(1+qY_1)(1+qY_3))\Psi_q(Y_3)\Psi_q(Y_1)\\
&=\Psi_q((1+qY_2 X_q)\inv Y_1)\Psi_q((1+qY_2X_q)\inv Y_3)\\
& \qquad
\Psi_q( (1+qY_2+Y_1Y_2)\inv (1+q^3Y_2+Y_3Y_2)\inv Y_3Y_2Y_1) \\
&\qquad \Psi_q(q(1+qY_2)\inv Y_2 Y_3)\Psi_q(q (1+qY_2)\inv Y_1 Y_2)\Psi_q(Y_2).
\end{split}
\end{align}

\end{Ex}

\subsection{Modular Double}

Let $b\in \C$ and write $q=e^{i \pi b^2}$, and define a new variable $\til{q}$ 
by $\til{q}=e^{i \pi b^{-2}}$.

\begin{Def}[Non-Compact Quantum Dilogarithm \cite{Faddeev:1994fw,Faddeev:1995nb}]
We define the non-compact quantum dilogarithm function $\Phi_b(x)$ by an integral\footnote{
This is also denoted by $e_b(x)$ in the literature. Note that this is the inverse of $\Phi_b$ used in \cite{KashaevN}.
}
\begin{align}
\Phi_b(x):= \exp\left(\int_{\bR+i0}\frac{e^{-2ixt}}{4\sinh(bt)\sinh(b^{-1}t)}\frac{dt}{t}\right)
\end{align}
in the strip $| \mathrm{Im}(z) | < \big| \mathrm{Im} \left(\frac{ i (b+b^{-1})}{2} \right)\big|$.
This function can then be analytically continued into the whole complex plane.
For $\mathrm{Im}(b^2)>0$, we have $|q|, |\til{q}\inv|<1$ and it turns out that $\Phi_b(x)$ can be written as an infinite product
\begin{align}
\Phi_b(x)=\frac{\prod_{k=0}^\oo (1+q^{2k+1}e^{2\pi b x})}{\prod_{k=0}^\oo (1+\til{q}^{-(2k+1)}e^{2\pi b\inv x})}= \frac{\Psi_{\til{q}^{-1}}(e^{2\pi b^{-1} x})}{\Psi_q(e^{2\pi b x})}.
\label{Psi_as_ratio}
\end{align}

\end{Def}

In the following we will also use a function $g_b(u)$, 
which is related to the function $\Phi_b(x)$ by the relation
$g_b(e^{2\pi bx})=\Phi_b(x)$.
Let $x,p$ with $[x,p]=\frac{1}{2\pi i}$ be the standard position and momentum operators acting on $L^2(\R)$. Also denote by 
\Eq{u=e^{2\pi bx}, \quad v=e^{2\pi bp},} 
such that 
\Eq{uv=q^2 vu} on a natural dense core of their domain \cite{Schmudgen}. Then the functions $\Phi_b(x), g_b(u)$ satisfy the pentagon identity as unitary operators on $L^2(\R)$:\footnote{The identity for $\Phi_b(x)$ is called the local form in \cite{KashaevN}.}
\begin{align}
\begin{split}
&\Phi_b(p)\Phi_b(x) = \Phi_b(x)\Phi_b(x+p)\Phi_b(p),
\\
&g_b(v)g_b(u) = g_b(u)g_b(qvu)g_b(v).
\label{Phi_gb_pentagon}
\end{split}
\end{align}

The concept of modular double is introduced in \cite{Faddeev:1995nb} and has been considered in various contexts of mathematics and physics, in particular to representation theory of split real quantum groups \cite{FrenkelIp, Ip2} and quantum Liouville theory \cite{PonsotTeschner1, BytskoTeschner} involving the Weyl operators. The idea is that the Weyl pair $\{u,v\}$ has a ``modular double counterpart" which mutually commutes with them, and it is natural to consider both pairs of Weyl operators in their functional analysis. 

Given a Weyl pair $\{u,v\}$ as above, their modular double counterpart is defined through the use of functional calculus on positive (essentially) self-adjoint operators
\Eq{\til{u}:=u^{\frac{1}{b^2}}=e^{2\pi b\inv x},\quad \til{v}:=v^{\frac{1}{b^2}}=e^{2\pi b\inv p}, }
so that \Eq{\til{u}\til{v}=\til{q}^2\til{v}\til{u}}
and consequently the pair $\{\til{u}, \til{v}\}$ mutually commute with $\{u,v\}$ on certain dense domain. With this notation, we see that
\begin{align}
g_b(u)= \frac{\Psi_{\til{q}^{-1}}(\til{u})}{\Psi_q(u)},
\label{gb_as_ratio}
\end{align}
and in particular it is manifest that we have the modular duality:
\Eq{g_b(u) = g_{b\inv}(\til{u}).}

Following the same idea, we introduce the modular dual variables formally by
\Eq{\til{Y_i}:=Y_i^{\frac{1}{b^2}}}
such that 
\Eq{\til{Y}_i \til{Y}_j = \til{q}^{2b_{ji}} \til{Y}_j \til{Y}_i.}
and use them as the initial seed of the \emph{dual quantum $y$-variables}, defined below:

\begin{Def}[Dual Quantum $y$-variable]
Let us define the dual quantum cluster $y$-variables by 
replacing $q$ by $\til{q}$ in the definition of quantum $y$-variables (Definition \eqref{quantum_y_def}).
The mutation rule \eqref{quantum_y_mutate} is then replaced by
\Eq{\til{Y}_i'=\case{
\til{Y}_k\inv,&i=k,\\
\displaystyle \til{Y}_i\prod_{m=1}^{b_{ik}}(1+\til{q}^{2m-1}\til{Y}_k),& i\neq k, b_{ik}\ge 0,\\
\displaystyle \til{Y}_i\prod_{m=1}^{|b_{ik}|}(1+\til{q}^{2m-1}\til{Y}_k\inv)\inv,&  i\neq k, b_{ik}<0.
}
 \label{dual_quantum_y_mutate}
}
Moreover, the dual quantum $y$-variables $\til{Y}_i$ commute with the original quantum $y$-variables $Y_i$.
\end{Def}

Since $Y_i$'s and their duals $\til{Y}_i$'s commute, the ratio of compact dilogarithms in \eqref{gb_as_ratio} can be factorized, and
we immediately learn that we can derive the quantum dilogarithm identities (with reverse order) for $g_b(u)$.

\begin{Thm}[Identities for Non-Compact Quantum Dilogarithms \cite{KashaevN}] 
Under the assumptions of Theorem \ref{qdilog_tropical}, we have, in tropical form
\Eq{g_b(Y_{\epsilon_L c_L})^{\epsilon_L} 
\ldots 
g_b(Y_{\epsilon_1 c_1} )^{\epsilon_1}=1,
\label{gb_id_tropical}
}
and in universal form
\Eq{g_b(Y_{k_1}(1)^{\epsilon_1})^{\epsilon_1}...g_b(Y_{k_L}(L)^{\epsilon_L})^{\epsilon_L}=1.
\label{gb_id_universal}
}
\end{Thm}

\section{Cyclic Dilogarithm Identities (I)}\label{sec.cyclic}

\subsection{Root of Unity Limit}

Let us choose a positive integer $N$.
In this section we will choose the value of $q$
to be $q=\zeta:=e^{\frac{\pi i}{N}}$, a primitive $2N$-th root of unity.
Then $q^2$ is the $N$-th primitive root of unity, and $q^N=-1$.
We also have 
\Eq{
\til{q}=\til{\zeta}:=e^{\pi iN} = (-1)^N.
}

With this specialization, the relation \eqref{quantum_y_comm} reduces to
\Eq{
Y_i Y_j= \zeta^{2 b_{ji}} Y_j Y_i,
\label{cyclic_Y_comm}
}
and
the dual quantum $y$-variable is simply given by
\Eq{
\til{Y}_i=Y_i^N,
}
which by definition commute with all the $Y_j$.

\begin{Prop}
When $\til{q}=\til{\zeta}=e^{\pi iN} = (-1)^N$, the transformation property of the dual quantum $y$-variable \eqref{dual_quantum_y_mutate}
reduces to 
\Eq{\til{Y_i}'=\case{
\til{Y}_k\inv,&i=k,\\
\displaystyle \til{Y_i}(1+\til{\zeta}\til{Y}_k)^{b_{ik}},& i\neq k, b_{ik}\ge 0,\\
\displaystyle \til{Y_i}(1+\til{\zeta}\til{Y}_k\inv)^{b_{ik}},&  i\neq k, b_{ik}<0.
}
\label{cyclic_Y_comm_2}
}
\end{Prop}

\begin{proof} 

When $\til{q}=\til{\zeta}= (-1)^N$,
we have
$$\til{q}^{2m-1} = (-1)^{(2m-1)N} = (-1)^N = \til{\zeta}.$$
Hence the product $\prod_{m=1}^{|b_{ik}|}$ in \eqref{dual_quantum_y_mutate} becomes trivial, leading to the formula above.

\end{proof}

\subsection{Identities for $d_N(x)$}

\begin{Def}[Cyclic $y$-variable]\label{cyclic_y_def}
Let us define the cyclic $y$-variables $\cY_i$.
They satisfy the commutation relation
\begin{align}
\cY_i \cY_j=\zeta^{2b_{ji}} \cY_j \cY_i ,
\end{align}
which coincides with the relation \eqref{quantum_y_comm}
for the quantum $y$-variables.
We also take the initial value to be $\cY_i=Y_i$.
However, we require $\cY_i$ to transform formally as the dual quantum $y$-variable
$\cY_i = \til{Y_i}^{\frac{1}{N}}$, so that we define

\Eq{\cY_i'=\case{
\cY_k\inv,&i=k,\\
\displaystyle \cY_i(1+\til{\zeta}\cY_k^N)^{\frac{b_{ik}}{N}},& i\neq k, b_{ik}\ge 0,\\
\displaystyle \cY_i(1+\til{\zeta}\cY_k^{-N})^{\frac{b_{ik}}{N}},&  i\neq k, b_{ik}<0.
}
 \label{cyclic_y_mutate}
}

We can repeat this for a mutation sequence, to obtain 
cyclic $y$-variables. Similar to the previous cases, we define
$(B(1), \cY(1)):=(B, \cY)=(B, Y(1))$ and
$(B(t+1), \cY(t+1)):=\mu_{k_t}(B(t), \cY(t))$.
We can also extend the definition of $\cY_i$ to $\cY_{\alpha}$ with $\alpha\in \bZ^{|I|}$,
as in \eqref{quantum_torus} in the case of $Y_i$.
\end{Def}

\begin{Def}[Cyclic Dilogarithm Function]\label{cyclicdilog}
For a given positive integer $N$, we define the \emph{cyclic dilogarithm function} $d_N(x)$ by\footnote{
Another definition found in the literature is 
\begin{align}
d^*_N(x)&:=
\left(1-x ^N\right)^{\frac{N-1}{2N}} \prod_{k=1}^{N-1} (1- \zeta^{2k} x)^{-\frac{k}{N}}.
\end{align}
This is related to our $d_N(x)$ by $d_N(x)=d^*_N(-\zeta x)$.
}
\begin{align}
d_N(x)&:=
\left(1+(-x)^N\right)^{\frac{N-1}{2N}} \prod_{k=1}^{N-1} (1+\zeta^{2k+1} x)^{-\frac{k}{N}}.
\label{cyclic_dilog_def}
\end{align}
\end{Def}

We are now ready to state one of our main theorems.

\begin{Thm}[Cyclic Dilogarithm Identity in Dual Universal Form]
Under the assumptions of Theorem \eqref{qdilog_tropical}, we have
\Eq{
d_N(\cY_{k_L}(L)^{\e_L})^{\e_L} \cdots d_N(\cY_{k_1}(1)^{\e_1})^{\e_1} = 1.
\label{cyclic_dilog_id_eq}
}
\label{cyclic_dilog_id}
\end{Thm}

\begin{proof}
The rest of this subsection is devoted to a proof of this theorem.

Following \cite{BazhanovR}, the basic strategy is to take the limit of the quantum dilogarithm identity for $\Psi_q(x)$. First of all let us consider the classical limit by taking 
\Eq{ q:=e^{-\frac{\tau}{2}},\quad \tau\in\R_{>0}}
so that $|q|<1$, and we take the limit when $\tau\to 0$. Then it is well-known that
\Eq{\Psi_q(x)^{-1}=(-qx;q^2)_{\infty} = R_{\tau}(x) \left(1+\mathcal{O}(\tau)  \right),
\label{R_def}
}
with $R_{\tau}(x)$ defined by
\begin{align}
R_{\tau}(x):= e^{-\frac{\textrm{Li}_2(-x)}{\tau}}.
\end{align}

The quantum dilogarithm identity in the tropical form \eqref{qdilog_tropical}
implies the following
\begin{Prop} \label{R-tropical} Let the quantum $Y$-variable be associated to $q=e^{-\tau/2}$. Then we have
\begin{align}
\lim_{\tau\to 0} R_{\tau}(Y_{\e_Lc_L})^{\epsilon_L} \cdots R_{\tau}(Y_{\e_1c_1})^{\epsilon_1}=1.
\end{align}
\end{Prop}
This will be important in our later analysis. The limit here is analyzed for example in \cite{KashaevN}.
As proven there, the condition that the resulting expression is trivial gives the 
classical dilogarithm identity.

Now let us consider the case of the root of unity limit. We let
\Eq{q:= e^{-\frac{\tau}{2N^2}} \zeta,}
and take the limit as $\tau\to 0$. In this limit, the quantum dilogarithm $\Psi_q(x)$ behaves as \cite{BazhanovR}
\footnote{
This is equivalent to the formula
\begin{align}
(x;q^2)_{\infty} = (1-x^N)^{\frac{1-N}{2N}}R^*_{\tau}(x^N) \, d^*_N(x) \left(1+\mathcal{O}(\tau)  \right),
\label{eq.qdilog_root_limit_another}
\end{align}
with
\begin{align}\label{Rstar}
R^*_{\tau}(x)&:=(1-x)^{\frac{1}{2}} e^{-\frac{\textrm{Li}_2(x)}{\tau}}.
\end{align}
}
\begin{align}
\Psi_q(x)^{-1}=(-q x;q^2)_{\infty} = R_{\tau}\left((-x)^N\right) \, d_N(x) \left(1+\mathcal{O}(\tau)  \right),
\label{eq.qdilog_root_limit}
\end{align}
with $d_N(x)$ defined previously in \eqref{cyclic_dilog_def} and $R_{\tau}(x)$ in \eqref{R_def}. To save notations let us denote by 
\Eq{R_{\tau,N}(x):=R_{\tau}((-x)^N).\label{RtauN}}

Now, substituting the limiting form \eqref{eq.qdilog_root_limit} into \eqref{eq.qdilog_tropical},
we obtain an identity involving the functions $R_{\tau,N}(x)$ and $d_N(x)$:
\Eq{\lim_{\tau\to 0}R_{\tau,N}(Y_{\e_Lc_L})^{\e_L}d_N(Y_{\e_Lc_L})^{\e_L}\cdots R_{\tau,N}(Y_{\e_1c_1})^{\e_1}d_N(Y_{\e_1c_1})^{\e_1}=1.}
To simplify the resulting expression, we can shuffle the positions of $R_{\tau,N}(x)$,
using the following Lemma:

\begin{Lem}[Conjugation by $R_{\tau,N}$ \cite{BazhanovR}]\label{com} 
Let $Y_i Y_j = q^{2b_{ji}} Y_j Y_i$. 
We have
\begin{align}
\lim_{\tau\to 0} R_{\tau,N}(Y_i) \cdot Y_j \cdot R_{\tau,N}(Y_i)\inv = Y_j(1+\til{q}Y_i^N)^{-\frac{b_{ji}}{N}}.
\label{conjugate_R}
\end{align}
\end{Lem}

\begin{proof} 
We here provide a self-contained proof using functional calculus.
We can commute $u_j$ to the left and get
\Eqn{
\mathrm{(LHS)}&=Y_j \cdot\lim_{\tau\to 0} R_{\tau,N}(q^{2b_{ji}}Y_i)R_{\tau,N}(Y_i)\inv\\
&=Y_j \cdot\lim_{\tau\to 0} \frac{R_{\tau}(e^{-\tau b_{ji}/N}(-Y_i)^N)}{R_{\tau}((-Y_i)^N)}.
}
Now the limit can be computed directly by calculus. Using $\Li_2(x) = -\int_0^x \frac{\log(1-t)}{t} dt$, and let $q(\tau) = e^{-\tau b_{ji}/N}$, we have
\begin{align}
\begin{split}
\lim_{\tau\to 0} \log \frac{R_{\tau}(q(\tau)(-x)^N)}{R_{\tau}((-x)^N)}
&=\lim_{\tau\to 0}\frac{1}{\tau}\left(\int_0^{-q(\tau)(-x)^N} \frac{\log(1-t)}{t} dt-\int_0^{(-x)^N} \frac{\log(1-t)}{t} dt\right)\\
&=\lim_{\tau\to 0}\frac{1}{\tau}\left(\int_{(-x)^N}^{-q(\tau)(-x)^N} \frac{\log(1-t)}{t} dt\right)\\
&=\lim_{\tau\to 0}\frac{\log(1-(-q(\tau)(-x)^N))}{-q(\tau)(-x)^N}\cdot \left(-\frac{dq(\tau)}{d\tau}(-x)^N\right)\\
&=-\log(1+(-x)^N) \frac{b_{ji}}{N}\\
&=-\log(1+\til{q}x^N)\frac{b_{ji}}{N}.
\label{diff_comp}
\end{split}
\end{align}
This proves \eqref{conjugate_R}.
\end{proof}

\begin{Rem}
In the limit $\tau\to 0$, $Y_i^N$, and hence any function of them, commute with all the $Y_j$.
However, in the analysis above the shuffling should first be performed for finite $\tau$, and
there remains a non-trivial finite contribution in the limit $\tau \to 0$
due to the canceling factor $1/\tau$ in 
the definition of $R_{\tau,N}(x)$.

\end{Rem}

Now the crucial observation is that the conjugation by $Y_i$ \eqref{conjugate_R}, taking into accounts the tropical signs $\e_t$, has exactly the same effect 
as the mutation rule for the cyclic $y$-variables \eqref{cyclic_y_mutate}.
In particular we can pull out all the $R_{\tau,N}$'s to the left or right,
as the cost of replacing the arguments of $d_N$'s by their
mutated versions. To be more precise, by mimicking the proof of 
\eqref{shuffle_id},
we obtain
\begin{align}
\cY_{k_t}(t) = \lim_{\tau\to 0}\mathrm{Ad}\inv
\left[
  R_{\tau,N}\left( Y_{\epsilon_{t-1} c_{t-1}} \right)^{\epsilon_{t-1}} \cdots
  R_{\tau,N}\left( Y_{\epsilon_t c_1} \right)^{\epsilon_1 }
\right]
\left( 
  \cY_{c_t}
\right) .
\end{align}
We can now appeal to the shuffling argument which we utilized in 
the proof of Theorem \ref{qdilog_id_universal}.
By shuffling the position of $R_{\tau, N}$, we obtain
(recall $Y_{\alpha}=Y_{\alpha}(1)=\cY_{\alpha}$)
\begin{align*}
&R_{\tau,N}\left(Y_{c_{L}\e_{L}} \right)^{\e_{L}} d_N(Y_{c_L\e_L})^{\e_{L}}  
   \cdots 
d_N(Y_{c_2\e_2})^{\e_{2}}  R_{\tau,N}\left(Y_{c_{1}\e_{1}} \right)^{\e_{1}} d_N(Y_{c_1\e_1})^{\e_{1}}  
        \\
&\quad=
R_{\tau,N}\left(Y_{c_{L}\e_{L}} \right)^{\e_{L}} d_N(Y_{c_L\e_L})^{\e_{L}}  
   \cdots  \\
& \qquad \qquad R_{\tau,N}\left(Y_{c_{1}\e_{1}} \right)^{\e_{1}} d_N(\cY_{k_2}(2)^{\e_2})^{\e_2} d_N(\cY_{k_1}(1)^{\e_1})^{\e_{1}}  
\left(1+\cO(\tau)\right).
\end{align*}
By repeating this, we arrive at
$$
R_{\tau,N}\left(Y_{c_L\e_L}\right)^{\e_L}\cdots R_{\tau,N}\left(Y_{c_1\e_1})\right)^{\e_1}d_N(\cY_{k_L}(L)^{\e_L})^{\e_L}...d_N(\cY_{k_1}(1)^{\e_1})^{\e_1}(1+\cO(\tau)) = 1.
$$

The remaining task is to show that the product of $R$'s vanish in the limit $\tau\to 0$. Recall
$R_{\tau,N}(Y_i):= R_{\tau}((-Y_i)^N).$
Let $Y_i' :=(-Y_i)^N$. Then 
$$Y_i'Y_j' = (-Y_i)^N(-Y_j)^N = q^{2N^2} (-Y_j)^N(-Y_i)^N = e^{-\tau}Y_j'Y_i',$$
hence all the arguments of $R$ satisfy the same commutation relation as the $Y_i$'s with $q=e^{-\tau/2}$. Hence by Proposition \ref{R-tropical}, we conclude that the product of $R$ is trivial. Since $d_N(x)$ does not depend on $\tau$ in the limit, we have proven the identity \eqref{cyclic_dilog_id_eq} for the function $d_N(x)$.
\end{proof}

\begin{Rem}
In this derivation, we shuffled all the $R$'s to the left of the product.
Alternatively we can choose to move $L'=0, \cdots, L$ of the $R$'s to the left and $L-L'$ to the right
,
and cancel the product of $R$'s in the middle.
We then obtain variants of the expression above, with
some of the variables rescaled according to \eqref{conjugate_R}.
\end{Rem}

We next rewrite the cyclic dilogarithm identity in another form.

\begin{Thm}[Cyclic Dilogarithm Identities in Standard Universal Form] We have the cyclic dilogarithm identities
\Eq{ d_N(Y_{k_1}(1)^{\e_1})^{\e_1}\cdots d_N(Y_{k_L}(L)^{\e_L})^{\e_L}=1.
\label{cyclic_id_standard_eq}}
for the standard quantum $y$-variables, with the reverse ordering.
\label{cyclic_id_standard}
\end{Thm}
\begin{proof} We start from the cyclic dilogarithm identity of \eqref{cyclic_dilog_id_eq},
we can mimic the arguments of the shuffle formula. First of all, we note that $d(Y_i)$ instead of conjugating tropical form to universal form, actually interchange universal form for $\cY_i$ and that for the standard $Y_i$ variables.

\begin{Lem} We have for $b_{ji}\geq 0$
\Eq{d_N(Y_i)\cdot Y_j \cdot d_N(Y_i)\inv=Y_j\frac{(1+\til{q}Y_i^N)^{\frac{b_{ji}}{N}}}{\prod_{k=1}^{b_{ji}}(1+q^{2k-1}Y_i)},}
and for $b_{ji}<0$
\Eq{d_N(Y_i)\cdot Y_j \cdot d_N(Y_i)\inv=Y_j\frac{(1+\til{q}Y_i^{N})^{\frac{b_{ji}}{N}}}{\prod_{k=1}^{|b_{ji}|}(1+q^{-(2k-1)}Y_i)\inv}.}
\end{Lem}
\begin{proof}
First assume $b_{ji}>0$. We then have
\Eqn{
d_N(Y_i)\cdot Y_j \cdot d_N(Y_i)\inv&=Y_j \frac{d_N(q^{2b_{ji}}Y_i)}{d_N(Y_i)}\\
&=Y_j \frac{\prod_{k=1}^{N-1}(1+q^{2k+1}Y_i)^{\frac{k}{N}}}{\prod_{k=1}^{N-1}(1+q^{2k+1+2b_{ji}}Y_i)^{\frac{k}{N}}}\\
&=Y_j\frac{\prod_{k=1}^{N-1}(1+q^{2k+1}Y_i)^{\frac{k}{N}}}{\prod_{k=1+b_{ji}}^{N-1+b_{ji}}(1+q^{2k+1}Y_i)^{\frac{k-b_{ji}}{N}}}\\
&=Y_j\prod_{k=1}^{b_{ji}}(1+q^{2k+1}Y_i)^{\frac{k}{N}}\prod_{k=1+b_{ji}}^{N-1}(1+q^{2k+1}Y_i)^{\frac{b_{ji}}{N}}\prod_{k=N}^{N-1+b_{ji}}(1+q^{2k+1}Y_i)^{\frac{b_{ji}-k}{N}}\\
&=Y_j\prod_{k=1}^{b_{ji}}(1+q^{2k+1}Y_i)^{\frac{k}{N}}\prod_{k=1+b_{ji}}^{N-1}(1+q^{2k+1}Y_i)^{\frac{b_{ji}}{N}}\prod_{k=0}^{b_{ji}-1}(1+q^{2k+1}Y_i)^{-1+\frac{b_{ji}-k}{N}}.
}
Now note that
\Eqn{\prod_{k=1+b_{ji}}^{N-1}(1+q^{2k+1}Y_i)^{\frac{b_{ji}}{N}}
&=\frac{\prod_{k=0}^{N-1}(1+q^{2k+1}Y_i)^{\frac{b_{ji}}{N}}}{\prod_{k=0}^{b_{ji}}(1+q^{2k+1}Y_i)^{\frac{b_{ji}}{N}}}\\
&=\frac{(1+\til{q}Y_i^N)^{\frac{b_{ji}}{N}}}{\prod_{k=0}^{b_{ji}}(1+q^{2k+1}Y_i)^{\frac{b_{ji}}{N}}}.
}
Hence combining, most of the terms cancel out, and we are left with
\Eqn{
d_N(Y_i)\cdot Y_j \cdot d_N(Y_i)\inv&=Y_j\frac{(1+\til{q}Y_i^N)^{\frac{b_{ji}}{N}}}{\prod_{k=0}^{b_{ji}-1}(1+q^{2k+1}Y_i)}\\
&=Y_j\frac{(1+\til{q}Y_i^N)^{\frac{b_{ji}}{N}}}{\prod_{k=1}^{b_{ji}}(1+q^{2k-1}Y_i)} , 
}
which is precisely the ratio of the ``non-tropical terms".
The case for $b_{ji}<0$ is similar.
\end{proof}

Therefore conjugation by $d_N$ will replace the non-tropical terms with its dual version. Now note that the non-tropical terms for $Y_i$ commute with $\cY_i$ and vice versa. Hence the only arguments that matter in the subsequent conjugations are from the tropical terms only. So the adjoint action follows from the shuffling formula as in the case of the $R$ function, but instead, with the tropical term replaced by the universal form of the dual variables.

In particular, we have
\Eqn{\textrm{Ad}\left[d_N(Y_{k_1}(1)^{\e_1})^{\e_1}...d_N(Y_{k_{t-1}}(t-1)^{\e_{t-1}})^{\e_{t-1}}\right](Y_{k_t}(t)) = \cY_{k_t}(t).
}
The rest of the proof is parallel to the proof of Theorem \ref{qdilog_id_universal}, see in particular \eqref{shuffle_id} and \eqref{shuffle_demonstration}.
\end{proof}

\begin{Rem}
The commutation relations \eqref{cyclic_Y_comm} and \eqref{cyclic_Y_comm_2}
has a finite-dimensional matrix representation, and consequently
the cyclic dilogarithm identity also reduces to a matrix identity.
For example, $Y_1 Y_2=\zeta^2 Y_2 Y_1$ (see Example \ref{A2_example_cyclic} below) can be represented by a  
$N\times N$ matrix
\begin{align}
(Y_1)_{ab}=\zeta^{2a} \delta_{ab}, \quad
(Y_2)_{ab}=\delta_{a, b+1},
\end{align}
where the indices $a,b$ are to be understood modulo $N$.
\end{Rem}

\subsection{Examples}

For concreteness let us now discuss two examples.

\begin{Ex}[$A_2$ quiver]\label{A2_example_cyclic}

Let us consider the mutation sequence discussed in Example \ref{A2_example_qdilog}.
The identity in the dual universal form \eqref{cyclic_dilog_id} reads
\begin{align}
\begin{split}
&d_N\left(Y_2(1+\til{\zeta} \til{Y}_1)^\frac{1}{N}\right) d_N(Y_1) \\
&\quad= d_N\left(\frac{Y_1}{(1+\til{\zeta} \til{Y}_2+\til{Y}_1 \til{Y}_2)^{\frac{1}{N}}}\right)d_N\left(\frac{\zeta Y_2 Y_1}{(1+\til{\zeta} \til{Y}_1)^\frac{1}{N}}\right)d_N(Y_2)
\end{split}
\end{align}
for $Y_1 Y_2=\zeta^2 Y_2 Y_1$.
One can verify that this coincides with the identity derived in \cite{BazhanovR}.
The identity in the standard form \eqref{cyclic_id_standard_eq} reads
\begin{align} 
\begin{split}
&d_N(Y_1) d_N(Y_2(1+qY_1)) d_N\left((1+qY_2+Y_1Y_2)\inv Y_1\right)\inv  \\
&\qquad d_N\left(q(1+qY_2)\inv Y_2Y_1\right)\inv d_N(Y_2)\inv=1.
\end{split}
\end{align}

\end{Ex}

\begin{Ex}[$A_3$ quiver]\label{A3_example_cyclic}

For the mutation sequence for the $A_3$ quiver considered in Example \ref{A3_example_qdilog}, 
we have, for $Y_1 Y_2=\zeta^2 Y_2 Y_1, Y_3 Y_2=\zeta^2 Y_3 Y_2$,

The identity in the dual universal form \eqref{cyclic_dilog_id} reads
\begin{align}
\begin{split}
&d_N(Y_2\til{X}_{\til{\zeta}} )d_N(Y_3) d_N(Y_1)\\
&=d_N\left(Y_1(1+\til{\zeta} \til{Y}_2 \til{X}_{\til{\zeta}})^{-\frac{1}{N}}\right)
d_N\left(Y_3(1+\til{\zeta} \til{Y}_2\til{X}_{\til{\zeta}})^{-\frac{1}{N}}\right) \\
&\qquad d_N\left(Y_1Y_2Y_3(1+\til{\zeta}\til{Y}_2+\til{Y}_1 \til{Y}_2)^{-\frac{1}{N}}
 (1+\til{\zeta} \til{Y}_2+\til{Y}_3 \til{Y}_2)^{-\frac{1}{N}} \right) \\
&\qquad d_N\left(\zeta Y_2 Y_3(1+\til{\zeta}\til{Y}_2)^{-\frac{1}{N}}\right) d_N\left(\zeta Y_2 Y_1 (1+\til{\zeta}\til{Y}_2)^{-\frac{1}{N}}\right) d_N(Y_2),
\end{split}
\end{align}
where $\til{X}_{\til{\zeta}}:=(1+\til{\zeta} \til{Y}_1)^{\frac{1}{N}} (1+\til{\zeta} \til{Y}_3)^{\frac{1}{N}}$,
while the identity in the standard form \eqref{cyclic_id_standard_eq} is given by
\begin{align}
\begin{split}
&d_N(Y_1)   d_N(Y_3)  d_N(Y_2(1+qY_1)(1+qY_3))\\
&= d_N(Y_2)
 d_N(q (1+qY_2)\inv Y_1 Y_2)
  d_N(q(1+qY_2)\inv Y_2 Y_3) \\
& \qquad d_N( (1+qY_2+Y_1Y_2)\inv (1+q^3Y_2+Y_3Y_2)\inv Y_3Y_2Y_1) \\
& \qquad d_N((1+qY_2X_q)\inv Y_3)
d_N((1+qY_2 X_q)\inv Y_1).
\end{split}
\end{align}
with $X_q$ as in \eqref{Xq_def}.

\end{Ex}

\section{Cyclic Dilogarithm Identities (II)}\label{sec.cyclic2}

In this section we consider the root of unity limit of the 
quantum dilogarithm identities for $\Phi_b(x)$ and $g_b(x)$.

\subsection{Root of Unity Limit of $g_b(x)$}

Let $q=e^{\pi i\frac{M}{N}}$ where $M, N$ are coprime, such that $b=\sqrt{\frac{M}{N}}$. It turns out the integral expression for $\Phi_b(x)$ can be evaluated explicitly, which is discussed in a recent preprint \cite{Garoufalidis:2014ifa}. Let $Q=b+b\inv$.\footnote{In \cite{Garoufalidis:2014ifa} they denote by $c_b = \frac{iQ}{2}$.}

\begin{Prop}\label{Phi_cyclic}
Let $s=\sqrt{MN}$. Then
\begin{align}
\Phi_b\left(\frac{z}{2\pi s}-\frac{iQ}{2}\right) = \frac{e^{\frac{i}{2\pi s^2}\Li_2(e^z)}(1-e^z)^{1+\frac{iz}{2\pi s^2}}}{D_N(e^{z/N};q^2)D_M(e^{z/M};\til{q}^{-2})},
\end{align}
where the (Kashaev's) cyclic dilogarithm $D_N(x;q)$ is defined as
\Eq{
D_N(x;q):=\prod_{k=1}^{N-1}(1-q^kx)^{k/N}.
}
\end{Prop}

In particular, for $M=N=1$, we have\footnote{
This identity is well-known in the literature of supersymmetric gauge theories \cite{Jafferis:2010un}.
}
\begin{align}
\Phi_1(z) = e^{\frac{i}{2\pi}\Li_2(e^{2\pi z})}(1-e^{2\pi z})^{iz}.
\end{align}
For $M=1$ but $N$ general, we have
\begin{align}
\begin{split}
\Phi_\frac{1}{\sqrt{N}}\left(z-\frac{iQ}{2}\right) &= \frac{e^{\frac{i}{2\pi s^2}\Li_2(e^{2\pi sz})}(1-e^{2\pi sz})^{1+\frac{iz}{s}}}{D_N(e^{2\pi sz/N};q^2)}\\
&= \frac{(1-e^{2\pi sz})\Phi_1(sz)^{\frac{1}{s^2}}}{D_N(e^{2\pi sz/N};q^2)}.
\label{Phi_M1}
\end{split}
\end{align}

The proof of Proposition \ref{Phi_cyclic} can be found in \cite{Garoufalidis:2014ifa}. We observe from \eqref{Phi_M1} that the quantum dilogarithm function at $b=1$ plays a special role in the case of root of unity. Using the results from previous analysis of the cyclic dilogarithm, we derive new identities involving this special $\Phi_1(x)$ function. In the appendix we will present a new proof of \eqref{Phi_M1} based on the $\tau$ analysis.

\subsection{Identities for $\newg_{1,N}(u)$}
Let $u=e^{2\pi bz}$. Note that in \eqref{Phi_M1}, 
\Eq{\Phi_1(sz) = g_1(u^N).}
Hence we can rewrite  \eqref{Phi_M1} using $g_b(u)$ as
\Eq{g_\frac{1}{\sqrt{N}}(e^{-\pi i}q\inv u) = \frac{(1-u^N)g_1(u^N)^{\frac{1}{N}}}{D_N(u;q^2)}.}
(Note we need to keep track of $-1=e^{\pm \pi i}$ since $g_b(u)$ depends on $\log(u)$.)
Hence replacing $u\mapsto e^{\pi i}qu$, we have
\Eq{\label{g_root_expression}g_{\frac{1}{\sqrt{N}}}(u) = \frac{(1+\til{q}u^N)g_1(e^{\pi i}\til{q}u^N)^{\frac{1}{N}}}{D_N(-qu;q^2)}.}

Let us denote by
\Eq{
\newg_{1,N}(u):=\widehat{g}_1(u^N)^{\frac{1}{N}}(1+\til{q}u^N)^{\frac{1+N}{2N}}, \quad \widehat{g}_1(u):=g_1(e^{\pi i}\til{q} u),
}
Then
\Eq{g_\frac{1}{\sqrt{N}}(u) = d_N(u)\newg_{1,N}(u),}
where $d_N(u)$ is our cyclic dilogarithm defined in Definition \ref{cyclicdilog}. We see that the function $\newg_{1,N}(u)$ plays the role of $R_{\tau, N}(x)$ in our previous analysis. In fact we have

\begin{Thm}\label{g1_identity}
$\newg_{1,N}(u)$ satisfies the quantum dilogarithm identities in the tropical form:
\begin{align}
\newg_{1,N}(Y_{\e_1 c_1})^{\e_1} \cdots
\newg_{1,N}(Y_{\e_L c_L})^{\e_L} =1.
 \label{cyclic_g_id}
\end{align}
\end{Thm}

\begin{proof}
Since $g_{\frac{1}{\sqrt{N}}}(u)$ satisfy the quantum dilogarithm identities, we have
\Eqn{d_N(Y_{\e_1c_1})^{\e_1}\newg_{1,N}(Y_{\e_1c_1})^{\e_1}\cdots d_N(Y_{\e_Lc_L})^{\e_L}\newg_{1,N}(Y_{\e_Lc_L})^{\e_L}=1.}

Now the proof follows from the following lemma:

\begin{Lem} The commutation relation for $\newg_{1,N}(Y_i)$ with $Y_j$ is exactly the same as $R_{\tau,N}(Y_i)$ in the $\tau\to 0$ limit, i.e., it shuffles as the mutation for the cyclic $y$-variables:
\Eq{\newg_{1,N}(Y_i)\cdot Y_j \cdot \newg_{1,N}(Y_i)\inv = Y_j(1+\til{q}Y_i^N)^{-\frac{b_{ji}}{N}}.}
\end{Lem}
\begin{proof} We only need to consider $\widehat{g}_1$ since the extra factor commutes with everything. We have
\Eqn{
\widehat{g}_1(Y_i^N) Y_j &= e^{\frac{i}{2\pi}\Li_2(Y_i^N)}(1+\til{q}Y_i^N)^{-\frac{\log Y_i^N}{2\pi i}}Y_j\\
&=e^{\frac{i}{2\pi}\Li_2(Y_i^N)}Y_j(1+\til{q}v^N)^{-\frac{\log q^{2b_{ji}N} Y_i^N}{2\pi i}}\\
&=Y_je^{\frac{i}{2\pi}\Li_2(Y_i^N)}(1+\til{q}Y_i^N)^{-b_{ji}-\frac{\log Y_i^N}{2\pi i}}\\
&=Y_j(1+\til{q}Y_i^N)^{-b_{ji}} \widehat{g}_1(Y_i^N) ,
}
and hence
$$
\widehat{g}_1(Y_i^N)^{\frac{1}{N}}\cdot Y_j \cdot  \widehat{g}_1(Y_i^N)^{-\frac{1}{N}} =Y_j(1+\til{q}Y_i^N)^{-\frac{b_{ji}}{N}}.
$$
\end{proof}

We therefore obtain, after repetition of shuffling,
\begin{align*}
\newg_{1,N}(Y_{\e_1 c_1})^{\e_1} \cdots
\newg_{1,N}(Y_{\e_L c_L})^{\e_L} 
d_N(\cY_{k_L}(L)^{\e_L})^{\e_L}  \cdots 
d_N(\cY_{k_1}(1)^{\e_1})^{\e_1}=1.
\end{align*}
Since the product of $d_N$'s cancel thanks to Theorem \ref{cyclic_dilog_id},
we obtain \eqref{cyclic_g_id}.
\end{proof}

From the Lemma we see that $\newg_{1,N}(u)$ satisfies the shuffle relations for the cyclic 
$y$-variables, and consequently
we can derive
\begin{Cor}
$\newg_{1,N}(u)$ also satisfies the quantum dilogarithm identities in the dual universal form:
\Eq{\newg_{1,N}(\cY_{k_L}(L)^{\e_L})^{\e_L}\cdots \newg_{1,N}(\cY_{k_1}(1)^{\e_1})^{\e_1}=1.
\label{g_uN_universal}}
\end{Cor}

\begin{Rem} Using the functional equations for $g_b$, one can see that equation \eqref{g_root_expression} is also equivalent to the identity for $\zeta=e^{2\pi i/N}$:
\Eq{\prod_{k=1}^{N}g_{\frac{1}{\sqrt{N}}}(\zeta^k z) = g_1(e^{\pi i}\til{q}z^N),}
which is a generalization of the classical dilogarithm identity
\Eq{N\sum_{k=1}^{N}Li_2(\zeta^k x)=Li_2(x^N). }
\end{Rem}

Finally, similarly analysis can be done for general $q=e^{\pi i\frac{M}{N}}$ with $M,N$ coprime, using the fact that the commutation relation for $\newg_{1,N}(Y_i)$ is also true for the modular double variables due to the modular duality of $g_b(x)$. Let $b=\sqrt{\frac{M}{N}}$ and $u=e^{2\pi bz}$, and denote by
\Eq{\newg_{1,N,M}(u):=g_1(e^{\pi i(M+N)} u^N)^{\frac{1}{MN}}(1+e^{\pi i(M+N)}u^N)^{\frac{1}{2N}+\frac{1}{2M}},}
note that $u^N = \til{u}^M$. Then using the fact that $D_N(u;q^2)$ commutes with $D_M(\til{u};\til{q}^2)$, we have
\begin{Thm} \label{Main_2}
The function $\newg_{1,N,M}(u)$ satisfies the same quantum dilogarithm identities as $\newg_{1,N}(u)$,
namely we have \eqref{cyclic_g_id} 
with $\newg_{1,N}(u)$ replaced by $\newg_{1,N,M}(u)$.
\end{Thm}

\appendix

\section{Proof of \eqref{Phi_M1}}

In this subsection we present a new proof of \eqref{Phi_M1}, and as a consequence another proof for Theorem \ref{g1_identity}. Similar analysis can be used to prove Proposition \ref{Phi_cyclic} in its full generality for $q=e^{\pi i\frac{M}{N}}$ but we will omit it for simplicity.
\begin{proof}[Proof of \eqref{Phi_M1}]
By definition, we have for Im$(b^2)>0$
\Eqn{
g_b(x) &= \frac{(-qx;q^2)_\infty}{(-\til{q}^{-1} \til{x};\til{q}^{-2})_\infty},
}
or under rescaling
\Eqn{
g_b(e^{\pi i}q\inv x) =\frac{(x;q^2)_\infty}{(\til{x};\til{q}^{-2})_\infty}.
}
Again, note that we need to keep track of $-1=e^{\pm \pi i}$ since $g_b(x)$ involves $\log(x)$.

Now recall from \eqref{Rstar} we have
\begin{align}
\begin{split}
(x;q)_\infty &= (1-x^N)^{\frac{1-N}{2N}}R^*_{\tau}(x^N)d^*_N(x)(1+\cO(\tau))\\
&=\frac{R^*_{\tau}(x^N)}{D_N(x;q)}(1+\cO(\tau)),
\label{Poch_RD}
\end{split}
\end{align}
where $R^*_{\tau}(x) := (1-x)^{\frac{1}{2}}e^{-\Li_2(x)/\tau}$.

Let $q=e^{-\tau/2N^2}\zeta = e^{-\tau/2N^2}e^{\pi i/N}$. If we write $q=e^{\pi ib^2}$, then
$b^2 = \frac{2\pi N+i\tau}{2\pi N^2}$ and 
$$\til{q}\inv = e^{-\pi ib^{-2}} = e^{-\frac{2\pi^2N^2i}{2\pi N+i\tau}} = e^{-\pi iN-\tau/2 + i\tau^2/4\pi N+\cO(\tau^3)}.
$$
Let us now derive the asymptotics of $(x;\til{q}^{-2})_\infty$. We apply Euler-Maclaurin formula to its logarithm, following the similar idea from \cite{BazhanovR}. We obtain
\Eqn{(x;\til{q}^{-2})_\infty &= (1-x)^{\frac{1}{2}} e^\frac{\Li_2(x)}{\log(\til{q}^{-2})}\\
&=(1-x)^{\frac{1}{2}} e^{\frac{\Li_2(x)}{-\tau+\frac{i\tau^2}{2\pi N}+\cO(\tau^3)}}\\
&=(1-x)^{\frac{1}{2}} e^{\frac{\Li_2(x)}{-\tau}(1+\frac{\tau}{2\pi iN})} (1+\cO(\tau))\\
&=R^*_{\tau}(x)e^{-\frac{1}{2\pi iN}\Li_2(x)}(1+\cO(\tau)),
}
and hence
\Eq{g_\frac{1}{\sqrt{N}}(e^{\pi i}q^{-1}x) &= \lim_{\tau\to 0}\frac{(x;q^2)_\infty}{(\til{x};\til{q}^{-2})_\infty}\nonumber\\
&=\left.\frac{R^*_{\tau}(x^N)}{R^*_{\tau}(\til{x})}\right|_{\tau\to 0}\frac{e^{-\frac{1}{2\pi iN}\Li_2(\til{x})}}{D_N(x;q^2)} (1+\cO(\tau)).\label{g1N_ratio}
}
The limit
\Eqn{&\lim_{\tau\to 0}\frac{R^*_{\tau}(x^N)}{R^*_{\tau}(\til{x})}=\lim_{\tau\to 0}e^{-\Li_2(x^N)/\tau+\Li_2(\til{x})/\tau}
}
can be computed as (see \eqref{diff_comp} for a similar computation)
\Eqn{\lim_{\tau\to 0}-\Li_2(x)/\tau+\Li_2(\til{x})/\tau &=\lim_{\tau\to 0}\frac{1}{\tau}\left(\int_0^{x^N}\frac{\log(1-t)}{t}dt-\int_0^{\til{x}}\frac{\log(1-t)}{t}dt\right)\\
&=\lim_{\tau\to 0}-\frac{\log(1-x(\tau))}{x(\tau)} \frac{d\til{x}}{d \tau}\\
&=-\frac{\log(1-x^N)\log(x)}{2\pi i},
}
where we used 
$\til{x} = x^{\frac{1}{b^2}} =x^{\frac{2\pi N^2}{2\pi N+i\tau}} = x^{N+\frac{\tau}{2\pi i}+\cO(\tau^2)}$.
We therefore obtain
\Eq{\left.g_\frac{1}{\sqrt{N}}(e^{\pi i}q^{-1}x)\right|_{b^2\to \frac{1}{N}}&=\frac{e^{\frac{i}{2\pi i N}\Li_2(x^N)}(1-x^N)^{-\frac{\log(x)}{2\pi i}}}{D_N(x;q^2)}\nonumber\\
&=\frac{g_1(x^N)^{\frac{1}{N}}}{D_N(x;q^2)}\label{g1_ratio}.}
Shifting $x\to e^{-2\pi i}x$ to the other side of the branch cut, we obtain
\Eqn{\left.g_\frac{1}{\sqrt{N}}(e^{-\pi i}q^{-1}x)\right|_{b^2\to \frac{1}{N}}=\frac{e^{\frac{i}{2\pi i N}\Li_2(x^N)}(1-x^N)^{1-\frac{\log(x)}{2\pi i}}}{D_N(x;q^2)},
}
or in terms of $x=e^{2\pi bz}$:
\Eqn{\Phi_{\frac{1}{\sqrt{N}}}\left(z-\frac{iQ}{2}\right) = \frac{e^{\frac{i}{2\pi i N}\Li_2(e^{2\pi Nbz})}(1-e^{2\pi Nbz})^{1+ibz}}{D_N(e^{2\pi bz};q^2)},
}
which is identical to \eqref{Phi_M1}.
\end{proof}

As a consequence of this analysis, we provide another proof of Theorem \ref{g1_identity}.

\begin{proof}[Proof of Theorem \ref{g1_identity}] Again let $q=e^{\pi ib^2} = e^{-\tau/2N^2}e^{\pi i/N}$.
From equation \eqref{Poch_RD}, \eqref{g1N_ratio} and \eqref{g1_ratio}, we conclude that
\Eqn{
\Phi_1(sz)^{\frac{1}{N}}=g_1(e^{2\pi\sqrt{N}z})^{\frac{1}{N}} = \lim_{\tau\to 0} \frac{R^*_{\tau}(x^N)}{(\til{x};\til{q}^{-2})_\oo},
}
where $x=e^{2\pi bz}$.
Now we substitute $z\mapsto z-\frac{iQ}{2}$ so that $\til{x}\mapsto -\til{q}\inv \til{x}$, then by definition we have 
$$x^N=e^{2\pi bNz} \mapsto -e^{\tau/2N}(-x)^N(1+\cO(\tau)),$$
hence 

\Eqn{
R^*_{\tau}(x^N) &\to R^*_{\tau}(-e^{\tau/2N}(-x)^N) \\
&= R^*_{\tau}(-e^{\tau(1+N)/2N}e^{-\tau/2}(-x)^N) \\
&= (1+(-x)^N)^{\frac{1+N}{2N}}R^*_{\tau}(-e^{-\tau/2}(-x)^N)(1+\cO(\tau))\\
&= (1+\til{q}x^N)^{\frac{1+N}{2N}}R_{\tau,N}(x)(1+\cO(\tau)),}
where $R_{\tau,N}(x)$ is defined in \eqref{RtauN}, and we have used the identity
$$R^*_{\tau}(e^{\a\tau}x) = (1-x)^\a R^*_{\tau}(x)(1+\cO(\tau)).$$
Hence we have
$$\Phi_1\left(s\left(z-\frac{iQ}{2}\right)\right)^{\frac{1}{N}} = \lim_{\tau\to 0} (1+\til{q}x^N)^{\frac{1+N}{2N}} R_{\tau,N}(x)\Psi_{\til{q}\inv}(\til{x}).$$
Now note that $\til{x}$ commute with $x^N$ (this does not involve the $\tau$ limit), hence we can move the $R_{\tau,N}$ to one side and $\Psi_{\til{q}^{-1}}(\til{x})$ to the other side.
Then since both $R_{\tau,N}(x)$ and $\Psi_{\til{q}\inv}(\til{x})$ satisfy the quantum dilogarithm identities, we conclude that
\Eqn{(1+\til{q}x^N)^{-\frac{1+N}{2N}}\Phi_1\left(s\left(z-\frac{iQ}{2}\right)\right)^{\frac{1}{N}} &=(1+\til{q}x^N)^{-\frac{1+N}{2N}} g_1(e^{-\pi i}\til{q}\inv x^N)^{\frac{1}{N}}\\
&=(1+\til{q}x^N)^{\frac{1+N}{2N}} g_1(e^{\pi i}\til{q}x^N)^{\frac{1}{N}}\\
&=\newg_{1,N}(x)
}
also satisfies the quantum dilogarithm identities.
\end{proof}

\bibliographystyle{amsplain_abbr}
\bibliography{root}

\end{document}